\documentclass{article}
\usepackage{amssymb,amsmath,bm,geometry,graphics,graphicx,color}
\usepackage{amsfonts}

\def\pdt2{\partial_t^2}
\def\pdx2{\partial_x^2}

\newcommand{\normmm}[1]{{\left\vert\kern-0.25ex\left\vert\kern-0.25ex\left\vert #1
    \right\vert\kern-0.25ex\right\vert\kern-0.25ex\right\vert}}

\newcommand{\abs}[1]{\left\vert#1\right\vert}

\def\RR{{\mathbb{R}}}
\def\CC{{\mathbb{C}}}

\newtheorem{theo}{Theorem}[section]

\newtheorem{rem}[theo]{Remark}
\newtheorem{defi}[theo]{Definition}

\newtheorem{assum}[theo]{Assumption}
\newtheorem{prop}[theo]{Proposition}

\def\no{\noindent}

\title{Exponential energy-preserving methods for charged-particle dynamics in a  strong and constant magnetic
field}

\author{Bin Wang\,\footnote{School of Mathematical Sciences, Qufu Normal
University, Qufu  273165,  P.R.China; Mathematisches Institut,
University of T\"{u}bingen, Auf der Morgenstelle 10, 72076
T\"{u}bingen, Germany. E-mail:~{\tt wang@na.uni-tuebingen.de}}   }

\begin{document}
\maketitle

\begin{abstract}
In this paper, exponential energy-preserving methods are formulated
and analysed for solving charged-particle dynamics  in a strong and
constant magnetic field.  The resulting method  can exactly preserve
the energy of the  dynamics.  Moreover, it is shown that the
magnetic moment of the considered system   is nearly conserved over
a long time along this exponential energy-preserving method, which
is proved by using modulated Fourier expansions. Other properties of
the method  including symmetry  and convergence are also studied. An
illustrated  numerical experiment is carried out to demonstrate  the
long-time behaviour of the method.
\medskip

\no{Keywords:} charged-particle dynamics, exponential
energy-preserving methods, modulated Fourier expansions, long-time
conservation,  strong and constant magnetic field

\medskip
\no{MSC:} 65L06, 65P10, 78A35, 78M25.

\end{abstract}

\section{Introduction}\label{intro}
In this paper, we  derive and analyse   energy-preserving methods
for charged-particle dynamics in a strong and constant magnetic
field
 \begin{equation}\label{charged-particle sts-cons}
\begin{array}[c]{ll}
\ddot{x}=\dot{x} \times \frac{1}{\epsilon} B +F(x), \quad
x(t_0)=x_0,\quad \dot{x}(t_0)=\dot{x}_0,\ \ t\in[t_0,T],
\end{array}
\end{equation}
where  $x(t)\in \RR^3$ describes the position of a particle,   $B =
\nabla_x \times A(x)$ is a constant magnetic field with the vector
potential $A(x) = -\frac{1}{2}x \times B\in \RR^3$ and  $F(x) =
-\nabla_x U(x)$ is an electric field with the scalar potential
$U(x)$.  Following \cite{Hairer2018}, we are devoted to   the
situation of a small positive acaling parameter $0<\epsilon\ll 1$
and assume  that $\abs{B}\geq 1$ in the Euclidean norm.   The energy
of this dynamics is given by
\begin{equation}\label{energy of cha}
E(x,v)=\frac{1}{2}\abs{v}^2+U(x),
\end{equation}
where $v=\dot{x}$ is the velocity of the particle.
 Denote the constant
vector $B$ by
 $B=(B_1,B_2,B_3)^{\intercal}$ with $B_i \in \RR$ for $i=1,2,3.$ By
 the definition of the cross product, we obtain
$ \dot{x} \times  B =  \tilde{B} \dot{x}, $ where $\tilde{B}$ is a
skew symmetric matrix
$$\tilde{B}=\left(
                   \begin{array}{ccc}
                     0 & B_3 & -B_2 \\
                     -B_3 & 0 & B_1 \\
                     B_2 & -B_1 & 0 \\
                   \end{array}
                 \right).
$$
The system \eqref{charged-particle sts-cons} as well as $v=\dot{x}$
can be rewritten as
\begin{equation}\label{charged-sts-first order}
\begin{array}[c]{ll}
\dot{x}=v,\ \dot{v}=\frac{1}{\epsilon} \tilde{B}v+F(x).
\end{array}
\end{equation}
With the analysis given in \cite{Arnold97,Cary2009}, it is well
known that  the magnetic moment
\begin{equation}\label{momentum for B}
I(x,v)=\frac{\abs{v_{\perp}}^2}{2\abs{B  }}=\frac{1}{2 \abs{
\tilde{B}}^3} \abs{\tilde{B}v}^2
\end{equation}
 is an adiabatic invariant,
where  $v_{\perp}=\frac{v \times B }{\abs{B  }}$ is orthogonal to $B
$ and $|\tilde{B}|=\sqrt{B_1^2+B_2^2+B_3^2}.$ Recently, it has been
shown in \cite{Hairer2018} that this quantity is nearly conserved
over long time scales, which is proved by using  the technique of
modulated Fourier expansion with state-dependent frequencies and
eigenvectors.

Charged-particle dynamics have been received much attention for a
long time (see, e.g. \cite{Arnold97,Cary2009}) and
 many effective methods have been developed for solving this system.   The Boris method was presented in
\cite{Boris1970} and it was   researched further in
\cite{Ellison2015,Hairer2017-1,Qin2013}. Various other kinds of
methods have also been researched for  charged-particle dynamics,
such as volume-preserving algorithms   in \cite{He2015}, symplectic
or K-symplectic algorithms  in
\cite{He2017,Tao2016,Webb2014,Zhang2016}, and symmetric multistep
methods  in \cite{Hairer2017-2}.  It is worth mentioning that, more
recently, a variational integrator has been  studied in
\cite{Hairer2018} for solving charged-particle dynamics in a strong
magnetic field. In this paper we are interested in the formulation
and analysis of energy-preserving (EP) methods when applied to
charged-particle dynamics in a strong and constant magnetic field.

 For the energy-preserving
methods, there have been a lot of studies on  this topic. Many
different   EP methods have been presented and analysed, such as the
average vector field (AVF) method (see, e.g.
   \cite{Celledoni2010,Quispel08}), discrete gradient
methods  (see, e.g.  \cite{McLachlan14}),  Hamiltonian Boundary
Value Methods   (see, e.g. \cite{Brugnano2010}),   EP collocation
methods  (see, e.g. \cite{Hairer2010})
 and trigonometric/exponential EP
   methods  (see, e.g. \cite{Li_Wu(sci2016),wang2012-1,wubook2018}). Based on these work, we will derive   and analyse novel EP
integrators for charged-particle dynamics in a strong and constant
magnetic field. The  long time  magnetic moment conservation of the
new integrator will also be researched via its modulated Fourier
expansion. The technique of   modulated Fourier expansion  was
firstly given in \cite{Hairer00} and then it has been successfully
used in the study of long-time behaviour for
  numerical methods/differential equations (see, e.g.
\cite{Cohen15,Cohen08-1,Gauckler13,Hairer16,hairer2006,Sanz-Serna09}).

The rest of this paper is organized as follows. In Section
\ref{sec:Formulation}, we first formulate the scheme of the  method
and then prove that it is symmetric. In Section \ref{sec:num exp},
two main results concerning energy preservation and magnetic moment
preservation  are presented and a numerical experiment is reported
to support these results. The proofs of the two main results are
given in Sections \ref{sec:proof1}-\ref{sec:proof2}, respectively.
The concluding remarks of this paper are given in the last section.

\section{Formulation of the method} \label{sec:Formulation}In order to
effectively solve    the system \eqref{charged-sts-first order},
 we consider
the following method.
\begin{defi}
\label{scheme 2}  The exponential energy-preserving  method for
solving \eqref{charged-sts-first order} is defined by:
\begin{equation}\label{EAVF}
\left\{\begin{array}[c]{ll}x_{n+1}=x_{n}+
h\varphi_1(\frac{h}{\epsilon} \tilde{B}) v_{n}+h^2
\varphi_2(\frac{h}{\epsilon} \tilde{B})  \int_{0}^1
F\big(x_{n}+\sigma(x_{n+1}-x_{n}) \big) d\sigma,\\
v_{n+1}=e^{\frac{h}{\epsilon} \tilde{B}}v_{n}+h \varphi_{1}(
\frac{h}{\epsilon} \tilde{B}) \int_{0}^1
F\big(x_{n}+\sigma(x_{n+1}-x_{n}) \big) d\sigma,
\end{array}\right.
\end{equation}
where $h$ is a stepsize and the $\varphi$-functions are defined by
\begin{equation}
 \varphi_0(z)=e^{z},\ \ \varphi_k(z)=\int_{0}^1
e^{(1-\sigma)z}\frac{\sigma^{k-1}}{(k-1)!}d\sigma, \ \ k=1,2.
\label{phi}%
\end{equation}   We denote this method by  EEP.
\end{defi}

\begin{rem}It is noted that this kind of method belongs to  exponential integrators, which have been widely   developed
and researched for solving highly oscillatory systems (see, e.g.
 \cite{Hochbruck2010,Hochbruck2009,Li_Wu(sci2016),Mei2017,wang-2016,wu2017-JCAM}).
\end{rem}

\begin{prop}\label{symmetric thm}
The EEP integrator \eqref{EAVF}  is symmetric.
\end{prop}
\begin{proof}Exchanging $(x_{n},v_{n})\leftrightarrow (x_{n+1},v_{n+1})$ and $
h\leftrightarrow -h$ in \eqref{EAVF}  yields
\begin{equation}
 \left\{\begin{array}[c]{ll}x_{n}=x_{n+1}-
h\varphi_1(-\frac{h}{\epsilon} \tilde{B}) v_{n+1}+h^2
\varphi_2(-\frac{h}{\epsilon} \tilde{B}) \int_{0}^1
F\big(x_{n}+\sigma(x_{n+1}-x_{n}) \big)
d\sigma,\\
v_{n}=e^{-\frac{h}{\epsilon} \tilde{B}}v_{n+1}-h \varphi_{1}(-
\frac{h}{\epsilon} \tilde{B}) \int_{0}^1
F\big(x_{n}+\sigma(x_{n+1}-x_{n}) \big) d\sigma,
\end{array}\right.\label{formula 1}%
\end{equation}
where the following fact has been used $$\int_{0}^1
F\big(x_{n+1}+\sigma(x_{n}-x_{n+1}) \big) d\sigma=\int_{0}^1
F\big(x_{n}+\sigma(x_{n+1}-x_{n}) \big) d\sigma.$$   We thus obtain
$$v_{n+1}=e^{\frac{h}{\epsilon} \tilde{B}}v_{n}+he^{\frac{h}{\epsilon} \tilde{B}} \varphi_{1}(-
\frac{h}{\epsilon} \tilde{B})\mathcal{I}=e^{\frac{h}{\epsilon}
\tilde{B}}v_{n}+h\varphi_{1}( \frac{h}{\epsilon}
\tilde{B})\mathcal{I}$$ by considering   the second formula of
\eqref{formula 1}. Inserting this into the first formula of
\eqref{formula 1}, we get
\begin{equation*}
\begin{array}{rl}
x_{n+1}&=x_{n}+h\varphi_1(-\frac{h}{\epsilon} \tilde{B}) v_{n+1}-h^2
\varphi_2(-\frac{h}{\epsilon} \tilde{B})\mathcal{I}\\
&=x_{n}+h\varphi_1(-\frac{h}{\epsilon} \tilde{B})
e^{\frac{h}{\epsilon} \tilde{B}}v_{n}+h^2\big(\varphi_{1}(-
\frac{h}{\epsilon} \tilde{B})\varphi_{1}( \frac{h}{\epsilon}
\tilde{B})-\varphi_2(-\frac{h}{\epsilon}
\tilde{B})\big)\mathcal{I}\\
&=x_{n}+h\varphi_1(\frac{h}{\epsilon} \tilde{B})
 v_{n}+h^2 \varphi_2(\frac{h}{\epsilon} \tilde{B}) \mathcal{I}.
\end{array}
\end{equation*}
 Therefore, the energy-preserving  integrator \eqref{EAVF} is
 symmetric.
\end{proof}

\section{Main results and numerical experiment}\label{sec:num exp}
\subsection{Main results}
 Before presenting
the main results of this paper, we  need the following assumptions,
which have been considered in \cite{Hairer00,hairer2006}.
\begin{assum}\label{ass}
\begin{itemize}
\item    We consider the initial values $$x_0=\mathcal{O}(1),\ \ \ v_0=\mathcal{O}(1)$$  such that the
energy $E$ is bounded independently of $\epsilon$ along the
solution.

\item It is  assumed that the numerical solution   stays in a compact set.

\item We require a lower bound on the  stepsize $h\tilde{\omega} \geq c_0 >
0$ with $\tilde{\omega}=\frac{|\tilde{B}|}{\epsilon}$.

\item The  following numerical non-resonance condition is assumed to
be true
\begin{equation}
|\sin(\frac{1}{2}kh\tilde{\omega})| \geq c \sqrt{h}\ \ \mathrm{for} \ \ k=1,2,\ldots,N\ \   \mathrm{with} \ \ N\geq2,\label{numerical non-resonance cond}%
\end{equation}
  which imposes a restriction on $N$ for a given $h$ and $\tilde{\omega}$.
\end{itemize}
\end{assum}

We first give the result about the energy preservation of the  EEP
method.
\begin{theo}\label{energy pre thm} (\textbf{Energy preservation.})
  The EEP integrator \eqref{EAVF} preserves the
energy $E$ in \eqref{energy of cha} exactly, i.e.,
$$E(x_{n+1},v_{n+1})=E(x_{n},v_{n})\qquad \textmd{for} \qquad n=0,1,\ldots.$$
\end{theo}

Now, we  present  the   result about the long time magnetic moment
conservation of the  EEP method.
\begin{theo} \label{2 sym Long-time
thm} (\textbf{Magnetic moment conservation.}) Under the conditions
of Assumption \ref{ass}, the long time magnetic moment is nearly
conserved over long times by the  EEP method
  \begin{equation*}
\begin{aligned}
I(x_n,v_n)=I(x_0,v_0)+\mathcal{O}\Big(\frac{1}{\abs{\cos(\frac{1}{2}
h\tilde{\omega})}}h\Big)+\mathcal{O}(h),
\end{aligned}
\end{equation*}
where $0\leq nh\leq h^{-N+1}.$ The constants symbolized by
$\mathcal{O}$ depend on $N$,   the final time $T$ and the constants
in the assumptions, but are independent of $n,\ h,\ \epsilon$.
\end{theo}
\begin{rem}We note that it is easy to make a choice of   the stepsize $h$ such that  $\abs{\cos(\frac{1}{2}
h\tilde{\omega})}$ is not small. In this case,  an improved
  result concerning the long time magnetic moment conservation is
  obtained
  \begin{equation*}
I(x_n,v_n) =I(x_0,v_0)+\mathcal{O}(h)
\end{equation*}
for $0\leq nh\leq h^{-N+1}.$
\end{rem}


\subsection{Numerical experiment}
 \begin{figure}[ptb]
\centering
\includegraphics[width=6cm,height=3cm]{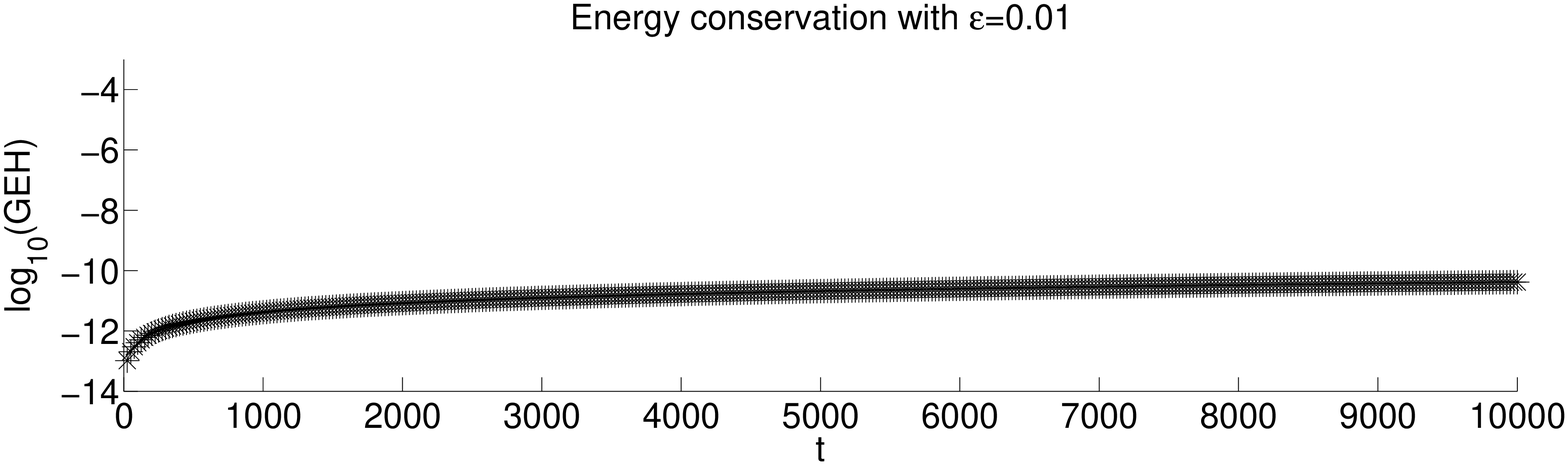}
\includegraphics[width=6cm,height=3cm]{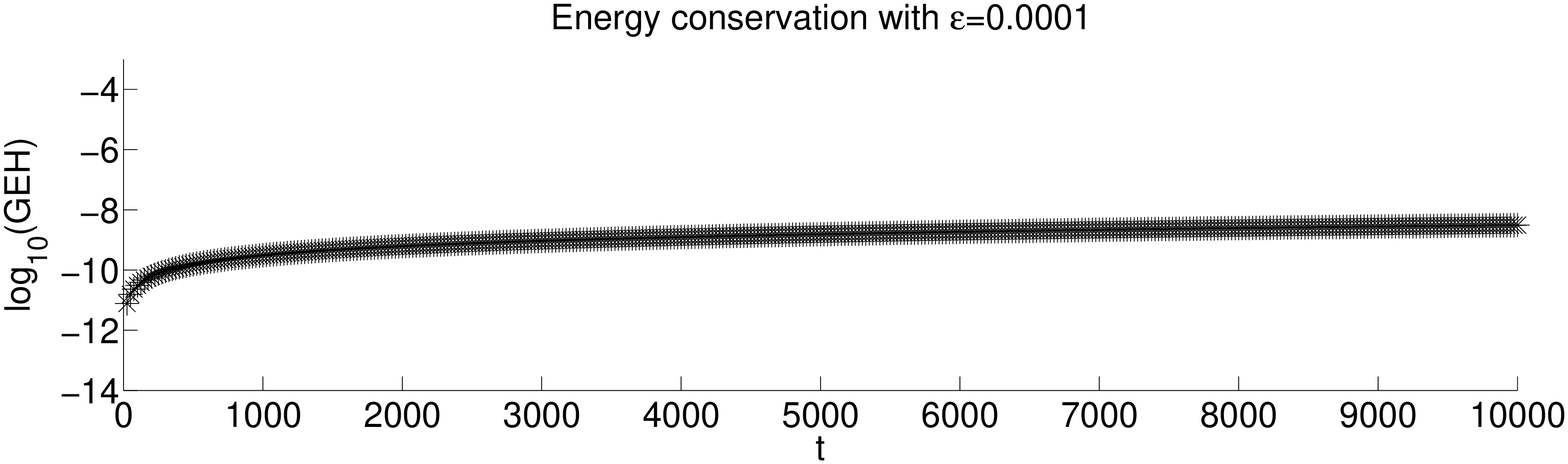}
\caption{the logarithm of the  energy relative  errors  against
$t$.} \label{p1}
\end{figure}
 \begin{figure}[ptb]
\centering
\includegraphics[width=6cm,height=3cm]{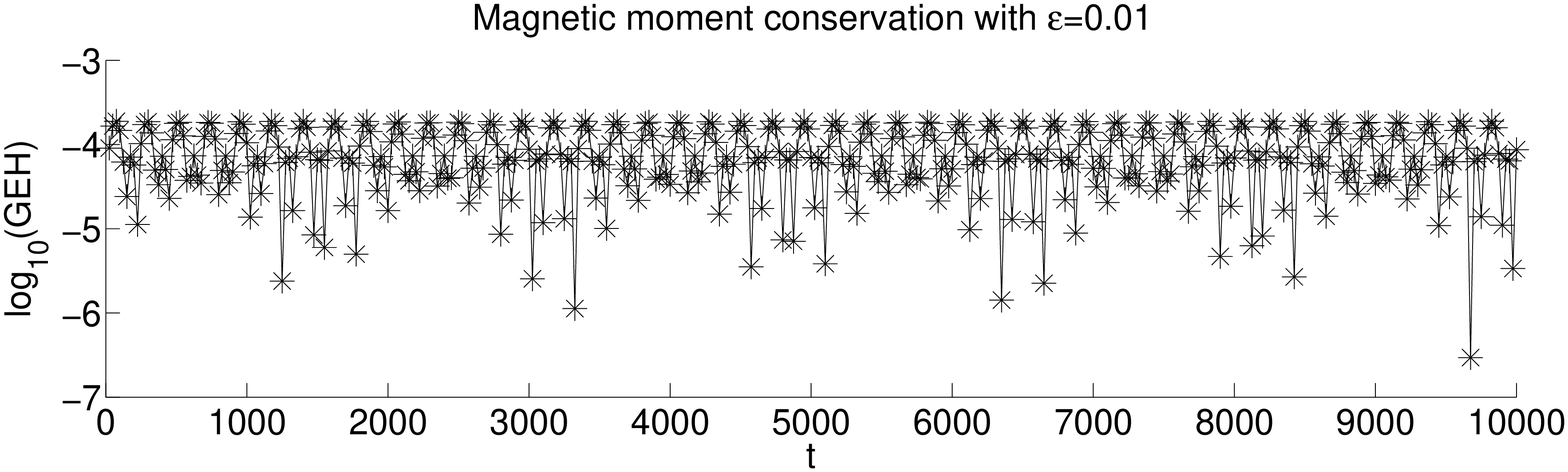}
\includegraphics[width=6cm,height=3cm]{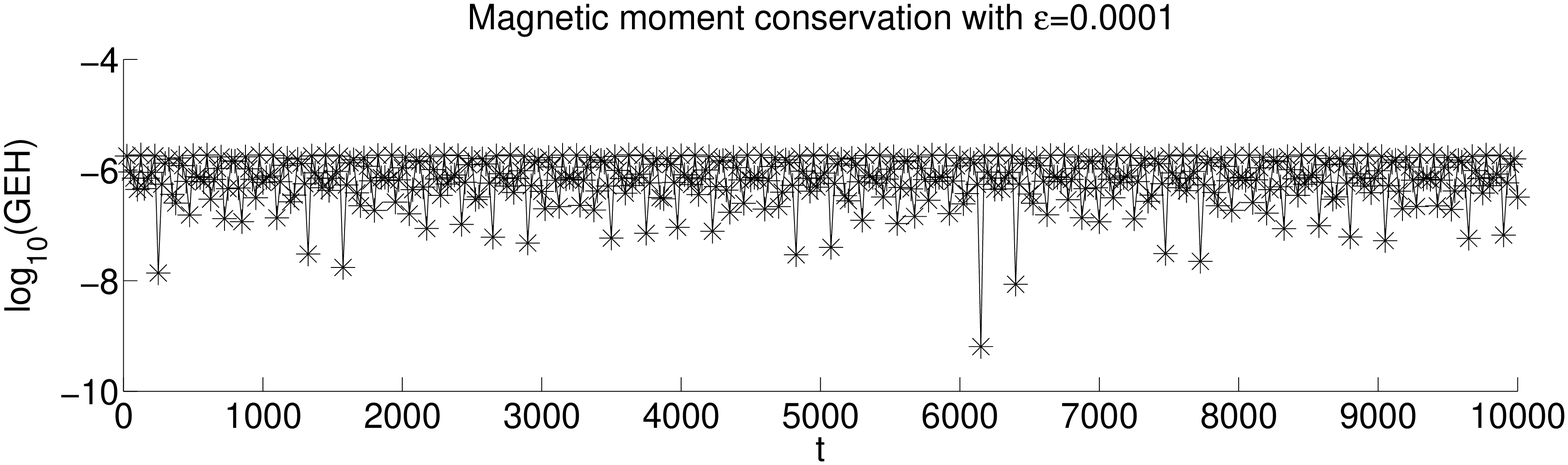}
\caption{the logarithm of the  magnetic moment relative errors
against $t$.} \label{p2}
\end{figure}
As an illustrative numerical experiment, we consider  the charged
particle system of \cite{Hairer2017-2} with a constant magnetic
field and  an additional factor $1/\epsilon$. The system can be
given by \eqref{charged-particle sts-cons} with the potential
$U(x)=\frac{1}{100\sqrt{x_1^2+x_2^2}}$ and the constant magnetic
field $B=(0,0,1)^{\intercal}.$  The initial values are chosen as
$x(0)=(0.7,1,0.1)^{\intercal}$ and $v(0)=(0.9,0.5,0.4)^{\intercal}.$

We consider applying
 four-point Gauss-Legendre's rule   to the integral of the  EEP integrator \eqref{EAVF}. The fixed-point iteration
 is chosen here and we set $10^{-16}$ as the error tolerance
 and $50$ as the maximum number of each iteration. We choose
 $\epsilon=0.01,0.0001$.
This problem is integrated on $[0,10000]$ with $h=0.01$ and see
Figures \ref{p1}-\ref{p2} for the relative errors
$(H(x_{n},v_{n})-H(x^0,v^0))/H(x^0,v^0)$ of the energy and
$(I(x_{n},v_{n})-I(x^0,v^0))/I(x^0,v^0)$  of the magnetic moment,
respectively. In order to show the performance of our EEP method, we
choose Boris method for comparison. We consider $\epsilon=0.0001$
and solve this system on $[0,10000]$ with $h=0.01$. The results of
Boris method are shown in Figure \ref{p3}. Finally, the problem is
solved with $\epsilon=0.05,$  $T=10,100, 1000$ and $h=
1/(50\times2^{i})$ for $i=0,\ldots,3$. The   global errors are shown
in Figure \ref{p4}.

 From the results, it can be  observed that our
EEP method   shows an excellent energy-preserving property, a
prominent long-term behavior in the numerical magnetic moment
conservation and a good accuracy. All these observations support the
theoretical results given in Theorems \ref{energy pre thm}- \ref{2
sym Long-time thm}.

 \begin{figure}[ptb]
\centering
\includegraphics[width=6cm,height=3cm]{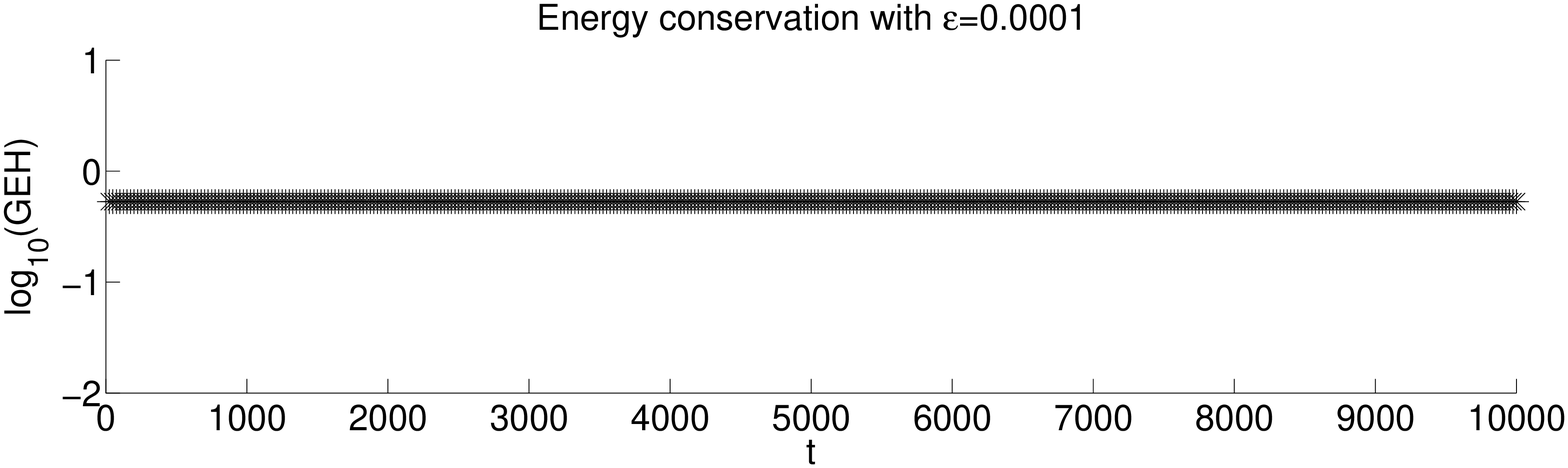}
\includegraphics[width=6cm,height=3cm]{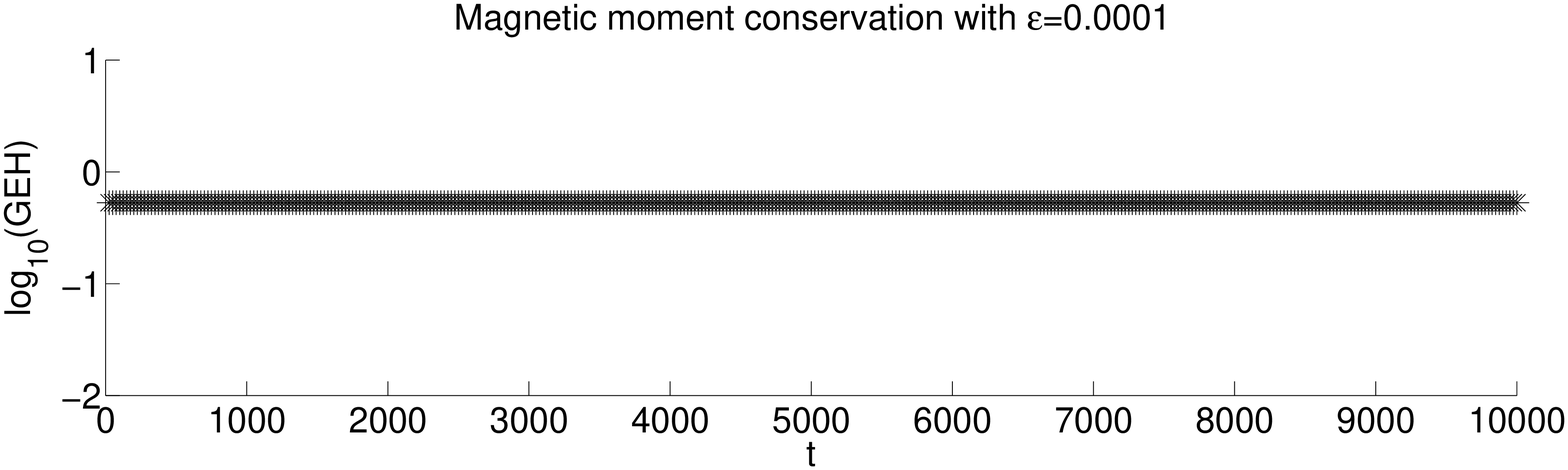}
\caption{the logarithm of the  relative  errors for Boris method
against $t$.} \label{p3}
\end{figure}

 \begin{figure}[ptb]
\centering
\includegraphics[width=4cm,height=5cm]{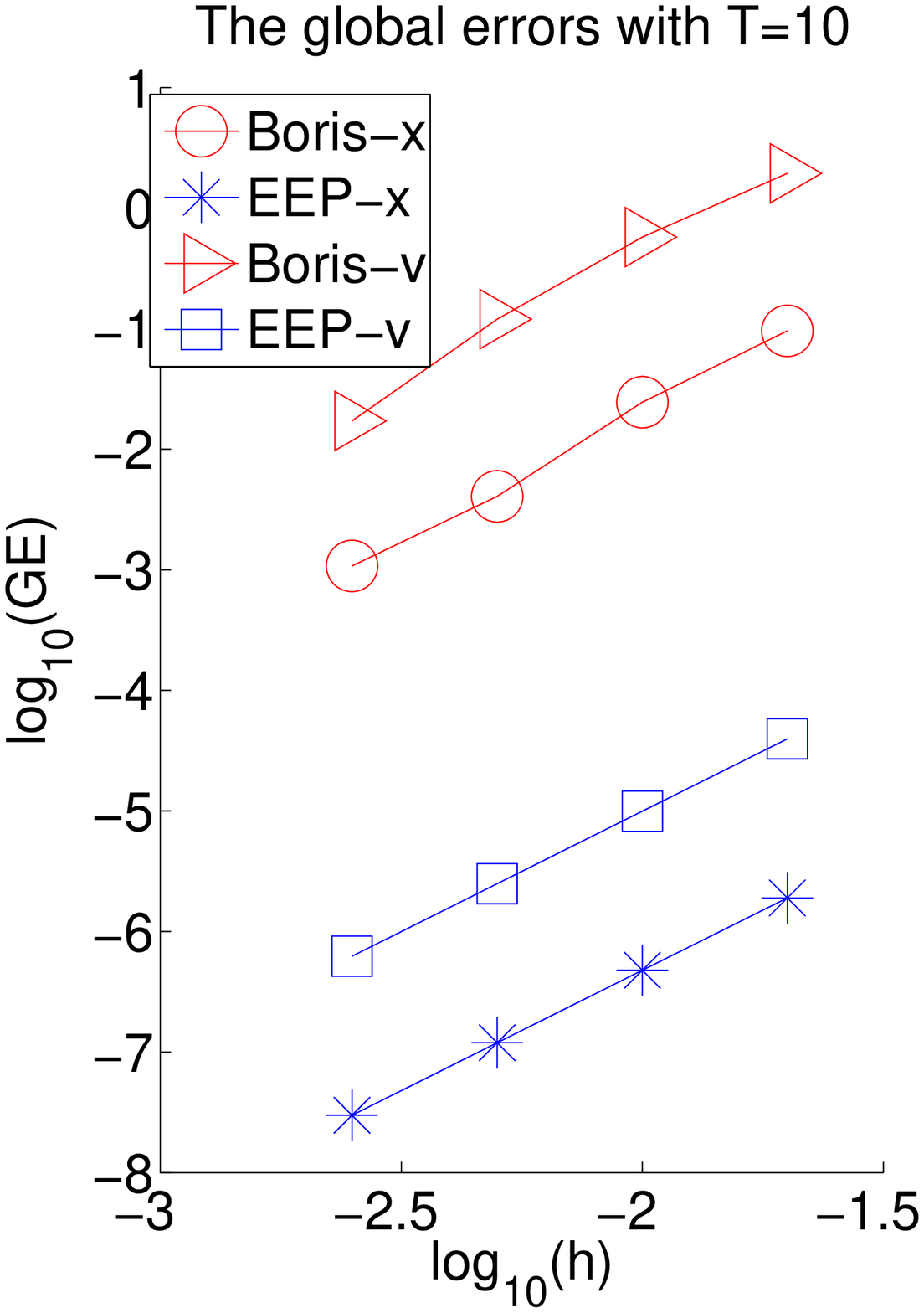}
\includegraphics[width=4cm,height=5cm]{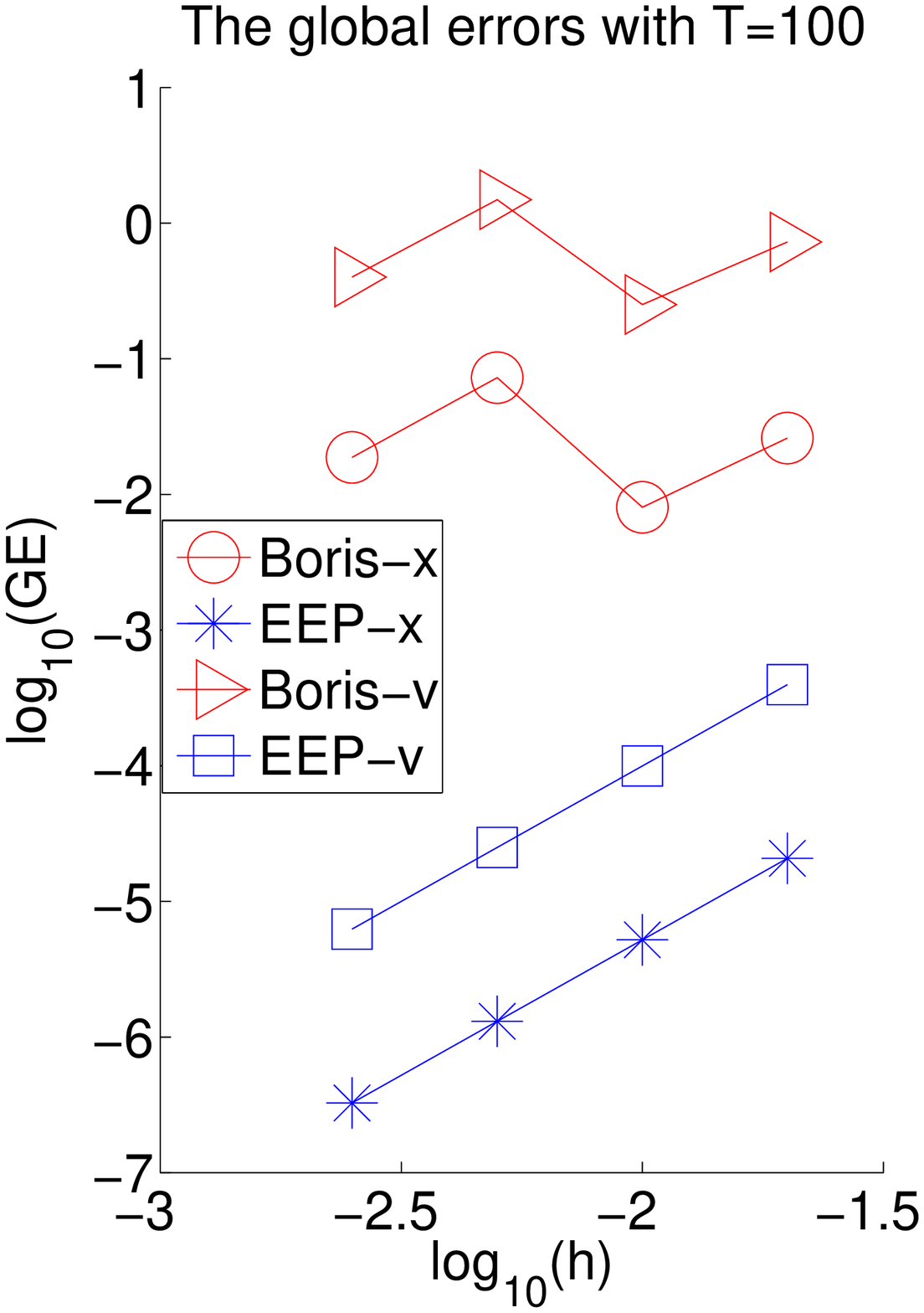}
\includegraphics[width=4cm,height=5cm]{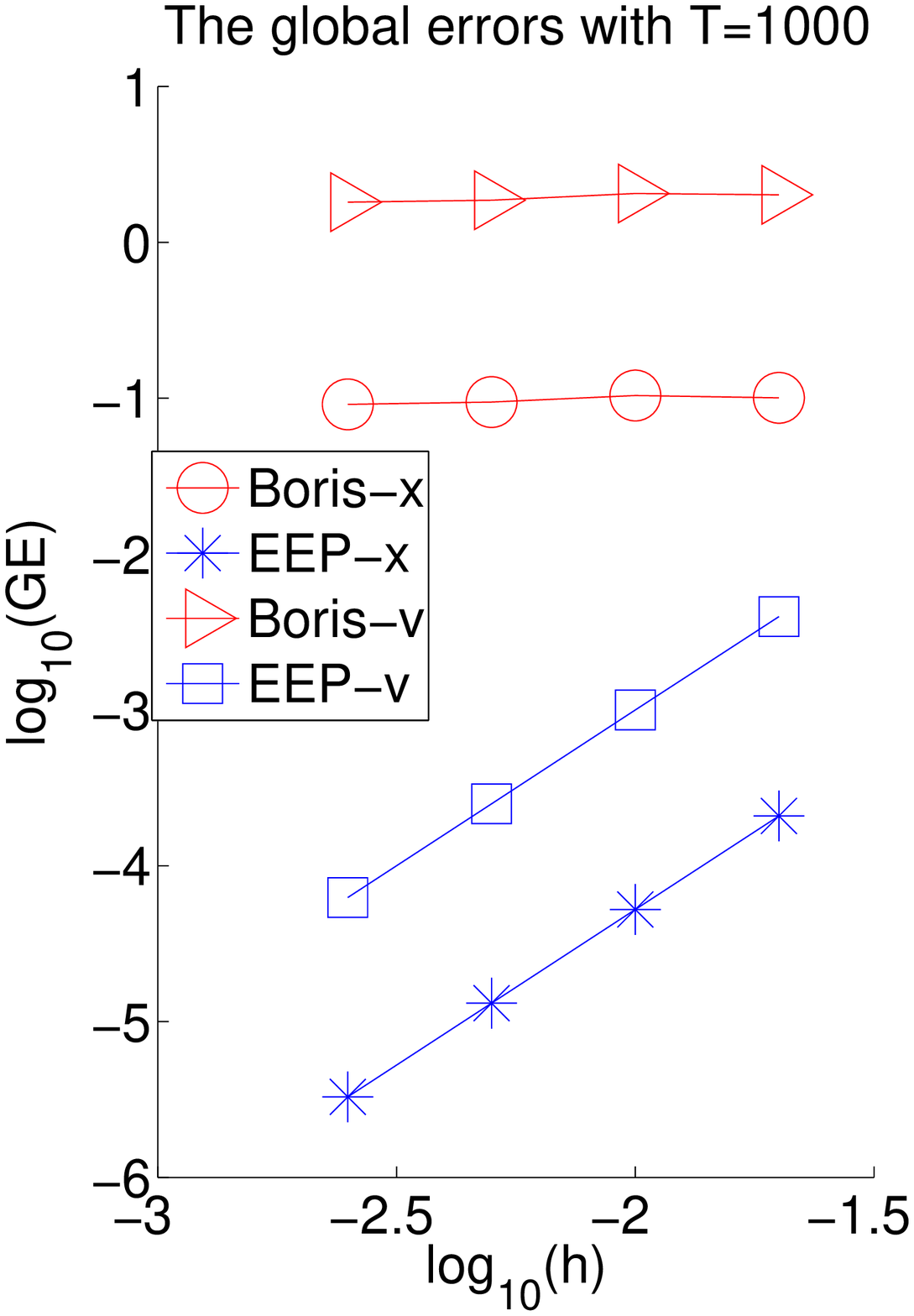}
\caption{the  global errors.} \label{p4}
\end{figure}

\section{Proof of energy preservation}\label{sec:proof1}
In this section, we give the proof of Theorem \ref{energy pre thm}.

\begin{proof}
In this paper, we  let $\mathcal{I}:=\int_{0}^1
F\big(x_{n}+\sigma(x_{n+1}-x_{n}) \big) d\sigma$ for brevity.  We
firstly compute
\begin{equation}
\begin{array}[c]{ll}E(x_{n+1},v_{n+1})=\dfrac{1}{2}v_{n+1}^{\intercal}v_{n+1}+U(x_{n+1}).\end{array}\label{for1}%
\end{equation}
Keeping  the fact in mind that $\tilde{B}$ is skew-symmetric, one
obtains that
$$(e^{\frac{h}{\epsilon} \tilde{B}})^{\intercal}=e^{-\frac{h}{\epsilon} \tilde{B}},\
(\varphi_{1}( \frac{h}{\epsilon}
\tilde{B}))^{\intercal}=\varphi_{1}(-\frac{h}{\epsilon} \tilde{B}),\
(\varphi_{2}( \frac{h}{\epsilon}
\tilde{B}))^{\intercal}=\varphi_{2}(-\frac{h}{\epsilon}
\tilde{B}).$$ Inserting the second formula of \eqref{EAVF} into
\eqref{for1} with some manipulation yields
\begin{equation*}
\begin{aligned}&E(x_{n+1},v_{n+1})
=\frac{1}{2}\big(e^{\frac{h}{\epsilon} \tilde{B}}v_{n}+h
\varphi_{1}( \frac{h}{\epsilon}
\tilde{B})\mathcal{I}\big)^{\intercal}
 \big(e^{\frac{h}{\epsilon} \tilde{B}}v_{n}+h
\varphi_{1}( \frac{h}{\epsilon} \tilde{B})\mathcal{I}\big)+U(x_{n+1})\\
=&\dfrac{1}{2}v_{n}^{\intercal} e^{-\frac{h}{\epsilon} \tilde{B}}
e^{\frac{h}{\epsilon} \tilde{B}} v_{n}+ h\mathcal{I}^{\intercal}
\varphi_{1}(-\frac{h}{\epsilon} \tilde{B})e^{\frac{h}{\epsilon}
\tilde{B}} v_{n} +\dfrac{1}{2}h^2\mathcal{I}^{\intercal}
\varphi_{1}(-\frac{h}{\epsilon}
\tilde{B})\varphi_{1}(\frac{h}{\epsilon} \tilde{B})\mathcal{I}
+U(x_{n+1}).\end{aligned}
\end{equation*}
 It follows from \eqref{phi}  that $\varphi_{1}(-\frac{h}{\epsilon}
\tilde{B})e^{\frac{h}{\epsilon}
\tilde{B}}=\varphi_{1}(\frac{h}{\epsilon} \tilde{B}).$ Therefore, we
obtain
\begin{equation}
\begin{aligned}E(x_{n+1},v_{n+1})
=\dfrac{1}{2}v_{n}^{\intercal}   v_{n}+ h\mathcal{I}^{\intercal}
\varphi_{1}(\frac{h}{\epsilon} \tilde{B})  v_{n}
+\dfrac{1}{2}h^2\mathcal{I}^{\intercal}
\varphi_{1}(-\frac{h}{\epsilon}
\tilde{B})\varphi_{1}(\frac{h}{\epsilon} \tilde{B})\mathcal{I} +U(x_{n+1}).\end{aligned}\label{for2}%
\end{equation}
 On the other hand, it can be checked that
\begin{equation*}
\begin{aligned}&U(x_{n})-U(x_{n+1})
=-\displaystyle\int_{0}^{1}dU((1-\tau)x_{n}+\tau x_{n+1})\\
=&-\displaystyle\int_{0}^{1}(x_{n+1}-x_n)^{\intercal}\nabla_x
U((1-\tau)x_{n}+\tau x_{n+1})d\tau
= \mathcal{I}^{\intercal}(x_{n+1}-x_n)\\
=&h\mathcal{I}^{\intercal} \varphi_1(\frac{h}{\epsilon} \tilde{B})
v_{n}+h^2\mathcal{I}^{\intercal} \varphi_2(\frac{h}{\epsilon}
\tilde{B})\mathcal{I}, \end{aligned}
\end{equation*}
where  the first equation of \eqref{EAVF} has been used. Inserting
this result  into \eqref{for2} implies
\begin{equation}
\begin{aligned}E(x_{n+1},v_{n+1})
=&\dfrac{1}{2}v_{n}^{\intercal}   v_{n}+ h\mathcal{I}^{\intercal}
\varphi_{1}(\frac{h}{\epsilon} \tilde{B})  v_{n}
+\dfrac{1}{2}h^2\mathcal{I}^{\intercal} \varphi_{1}(-\frac{h}{\epsilon} \tilde{B})\varphi_{1}(\frac{h}{\epsilon} \tilde{B})\mathcal{I} +U(x_{n})\\
&-h\mathcal{I}^{\intercal} \varphi_1(\frac{h}{\epsilon} \tilde{B})
v_{n}-h^2\mathcal{I}^{\intercal} \varphi_2(\frac{h}{\epsilon}
\tilde{B})\mathcal{I}\\
=&\dfrac{1}{2}v_{n}^{\intercal}   v_{n}+U(x_{n})
+\dfrac{1}{2}h^2\mathcal{I}^{\intercal}
\Big(\varphi_{1}(-\frac{h}{\epsilon}
\tilde{B})\varphi_{1}(\frac{h}{\epsilon}
\tilde{B})-2\varphi_2(\frac{h}{\epsilon}
\tilde{B})\Big)\mathcal{I} .\end{aligned}\label{for6}%
\end{equation}
It follows from the definition \eqref{phi} that
$\varphi_{1}(-\frac{h}{\epsilon}
\tilde{B})\varphi_{1}(\frac{h}{\epsilon}
\tilde{B})-2\varphi_2(\frac{h}{\epsilon}
\tilde{B})=\sum\limits_{k=1}^{\infty}c_k\big(\frac{h}{\epsilon}
\tilde{B}\big)^{2k+1}$ with the coefficients $c_k$ for
$k=1,2,\ldots$, which gives that
\begin{equation*}
\begin{aligned}&\Big(\varphi_{1}(-\frac{h}{\epsilon}
\tilde{B})\varphi_{1}(\frac{h}{\epsilon}
\tilde{B})-2\varphi_2(\frac{h}{\epsilon}
\tilde{B})\Big)^{\intercal}=\sum\limits_{k=1}^{\infty}c_k\big(-\frac{h}{\epsilon}
\tilde{B}\big)^{2k+1}\\
&=-\sum\limits_{k=1}^{\infty}c_k\big(\frac{h}{\epsilon}
\tilde{B}\big)^{2k+1}=-\Big(\varphi_{1}(-\frac{h}{\epsilon}
\tilde{B})\varphi_{1}(\frac{h}{\epsilon}
\tilde{B})-2\varphi_2(\frac{h}{\epsilon}
\tilde{B})\Big).\end{aligned}
\end{equation*} Based on
this result, one arrives at
\begin{equation*}
\begin{aligned} \mathcal{I}^{\intercal}
\Big(\varphi_{1}(-\frac{h}{\epsilon}
\tilde{B})\varphi_{1}(\frac{h}{\epsilon}
\tilde{B})-2\varphi_2(\frac{h}{\epsilon} \tilde{B})\Big)\mathcal{I}
&=\Big[\mathcal{I}^{\intercal} \Big(\varphi_{1}(-\frac{h}{\epsilon}
\tilde{B})\varphi_{1}(\frac{h}{\epsilon}
\tilde{B})-2\varphi_2(\frac{h}{\epsilon}
\tilde{B})\Big)\mathcal{I}\Big]^{\intercal}\\
&=-\mathcal{I}^{\intercal} \Big(\varphi_{1}(-\frac{h}{\epsilon}
\tilde{B})\varphi_{1}(\frac{h}{\epsilon}
\tilde{B})-2\varphi_2(\frac{h}{\epsilon}
\tilde{B})\Big)\mathcal{I},\end{aligned}
\end{equation*}
which shows that $\mathcal{I}^{\intercal}
\Big(\varphi_{1}(-\frac{h}{\epsilon}
\tilde{B})\varphi_{1}(\frac{h}{\epsilon}
\tilde{B})-2\varphi_2(\frac{h}{\epsilon}
\tilde{B})\Big)\mathcal{I}=0.$ Therefore, \eqref{for6} becomes
$$E(x_{n+1},v_{n+1})=\dfrac{1}{2}v_{n}^{\intercal}   v_{n}+U(x_{n})=E(x_{n},v_{n}).$$
\end{proof}

\section{Proof of the magnetic moment conservation}\label{sec:proof2}
This section is devoted to the proof of Theorem \ref{2 sym Long-time
thm}.   Modulated Fourier expansion will be used here.   We will
first derive the modulated Fourier expansion for the EEP method in
Subsection \ref{subs: mfe}  and then show an almost-invariant of the
expansion in Subsection \ref{subs: ai}. Based on these analysis,
Theorem \ref{2 sym Long-time thm} will be proved immediately.

Since the    matrix $\tilde{B}$ is  skew-symmetric,  there exists a
unitary matrix $P$ and a   diagonal matrix $\Lambda$ such that $
\tilde{B}=P \Lambda P^\textup{H}$ with $\Lambda=
\textmd{diag}(-|\tilde{B}|\mathrm{i},0,|\tilde{B}|\mathrm{i}) $. By
the linear change of variable
 \begin{equation}\label{change of variable}\tilde{x}(t)= P^\textup{H} x(t),\quad \tilde{v}(t)= P^\textup{H} v(t),\end{equation}
 we rewrite the system
\eqref{charged-sts-first order}  as
\begin{equation}\label{necharged-sts-first order}
\begin{array}[c]{ll}
  \dot{\tilde{x}}=\tilde{v}, &\tilde{x}_0= P^\textup{H} x_0,\\
 \dot{\tilde{v}}=\mathrm{i}\tilde{\Omega} \tilde{v}+
 \tilde{F}(\tilde{x}),\ &\tilde{v}_0= P^\textup{H} v_0,
\end{array}
\end{equation}
where $\tilde{\Omega}=
\textmd{diag}(-\tilde{\omega},0,\tilde{\omega})$ with
$\tilde{\omega}=\frac{|\tilde{B}|}{\epsilon}$ and
$\tilde{F}(\tilde{x})=P^\textup{H} F(P \tilde{x})=
-\nabla_{\tilde{x}} U(P\tilde{x})$. In this paper, a vector $x$ in
$\RR^3$ or $\CC^3$ is denoted by $x=(x_{-1},x_0,x_1)^{\intercal}$.
%
For the transformed  system \eqref{necharged-sts-first order}, its
 magnetic moment has the following form
\begin{equation}\label{newmomentum for B}
\begin{array}[c]{ll}
I(x,v)&=\frac{1}{2 \abs{ \tilde{B}}^3} \abs{\tilde{B}v}^2=\frac{1}{2
\abs{ \tilde{B}}^3} \abs{P \Lambda \tilde{v}}^2=\frac{1}{2 \abs{
\tilde{B}}^3}
\abs{  \Lambda \tilde{v}}^2\\
&=\frac{1}{2 \abs{ \tilde{B}}^3}\abs{ \tilde{B}}^2(
 \abs{ \tilde{v}_{-1}}^2+\abs{\tilde{v}_{1}}^2)=\frac{1}{2 \abs{ \tilde{B}} } (
 \abs{ \tilde{v}_{-1}}^2+\abs{\tilde{v}_{1}}^2):=\tilde{I}(\tilde{x},\tilde{v}).\end{array}
\end{equation}
 The EEP integrator
\eqref{EAVF} for solving this transformed  system   is defined as
\begin{equation}\label{ne exp integ one-stage}
\left\{\begin{array}[c]{ll}
&\tilde{x}_{n+1}=\tilde{x}_n+h\varphi_1(\mathrm{i}h\tilde{\Omega})
\tilde{v}_{n}+h^2 \varphi_2(\mathrm{i}h\tilde{\Omega}) \int_{0}^1
\tilde{F}\big(\tilde{x}_{n}+\sigma(\tilde{x}_{n+1}-\tilde{x}_{n}) \big) d\sigma,\\
&\tilde{v}_{n+1}=e^{
\mathrm{i}h\tilde{\Omega}}\tilde{v}_n+h\varphi_1(\mathrm{i}h\tilde{\Omega})
\int_{0}^1
\tilde{F}\big(\tilde{x}_{n}+\sigma(\tilde{x}_{n+1}-\tilde{x}_{n})
\big) d\sigma.
\end{array}\right.
\end{equation}

 In this section, the following   five operators will be used:
\begin{equation}\label{LLL}
\begin{aligned}L_1(hD):&=\mathrm{e}^{hD}-e^{
\mathrm{i}h\tilde{\Omega}},\\
L_2(hD):&=\varphi_1(\mathrm{i}h\tilde{\Omega})\mathrm{e}^{\frac{1}{2}hD},\\
L_3(hD):&=\varphi_1(\mathrm{i}h\tilde{\Omega})(\mathrm{e}^{hD}-1),\\
L_4(hD):&=h\varphi_2(\mathrm{i}h\tilde{\Omega}) \mathrm{e}^{hD}
+h\varphi^2_1(\mathrm{i}h\tilde{\Omega})-he^{
\mathrm{i}h\tilde{\Omega}}\varphi_2(\mathrm{i}h\tilde{\Omega}),\\
L_5(hD,\tau,k):&= (1-\tau)\mathrm{e}^{-\mathrm{i}
\frac{h}{2}k\tilde{\omega} }\mathrm{e}^{-
 \frac{h}{2}D} +\tau\mathrm{e}^{\mathrm{i} \frac{h}{2}k\tilde{\omega}} \mathrm{e}^{
 \frac{h}{2}D},\\
 L(hD):&=(L_1L_2^{-1} L_3L_4^{-1} )(hD),\end{aligned}
\end{equation} where $D$ is the
differential operator (see \cite{hairer2006}). We have the following
properties of these operators.
\begin{prop}\label{lhd pro}
 The Taylor expansions  of the operator $L(hD)$ are   expressed by
\begin{equation*}
\begin{aligned}&L(hD)= -\mathrm{i}\tilde{\Omega} hD+hD ^2+\cdots,\\
&L(hD+\mathrm{i}h\tilde{\omega})=\textmd{diag}\Big(-\frac{h\tilde{\omega}^2\sin(h\tilde{\omega})}{h\tilde{\omega}\cos(\frac{h\tilde{\omega}}{2})-\sin(\frac{h\tilde{\omega}}{2})},
\frac{-8\csc(h\tilde{\omega})\sin^3(\frac{h\tilde{\omega}}{2})}{h},0\Big)
+\textmd{diag}\\
&\Big(\frac{h^2\tilde{\omega}^2(4h\tilde{\omega}\cos^3(\frac{h\tilde{\omega}}{2})+\sin
(\frac{h\tilde{\omega}}{2})-\sin(\frac{3h\tilde{\omega}}{2}))}{4(\sin(\frac{h\tilde{\omega}}{2})-h\tilde{\omega}\cos(\frac{h\tilde{\omega}}{2}))^2},
(3+\cos(h\tilde{\omega}))\sec(\frac{3h\tilde{\omega}}{2})\tan(\frac{3h\tilde{\omega}}{2}),\frac{1}{2}h^2\tilde{\omega}^2\csc(\frac{h\tilde{\omega}}{2}) \Big)(\mathrm{i}D)+\cdots,\\
&L(hD-\mathrm{i}h\tilde{\omega})=\textmd{diag}\Big(
0,\frac{-8\csc(h\tilde{\omega})\sin^3(\frac{h\tilde{\omega}}{2})}{h},
-\frac{h\tilde{\omega}^2\sin(h\tilde{\omega})}{h\tilde{\omega}\cos(\frac{h\tilde{\omega}}{2})-\sin(\frac{h\tilde{\omega}}{2})}
\Big)+\textmd{diag}\\
&\Big(-\frac{1}{2}h^2\tilde{\omega}^2\csc(\frac{h\tilde{\omega}}{2}),
-(3+\cos(h\tilde{\omega}))\sec(\frac{3h\tilde{\omega}}{2})\tan(\frac{3h\tilde{\omega}}{2}),\frac{h^2\tilde{\omega}^2(4h\tilde{\omega}\cos^3(\frac{h\tilde{\omega}}{2})+\sin
(\frac{h\tilde{\omega}}{2})-\sin(\frac{3h\tilde{\omega}}{2}))}{-4(\sin(\frac{h\tilde{\omega}}{2})-h\tilde{\omega}\cos(\frac{h\tilde{\omega}}{2}))^2}\Big) (\mathrm{i} D)+\cdots,\\
&L(hD+\mathrm{i}kh\tilde{\omega})=-  8h\tilde{\Omega}^2
\sin(\frac{hk\omega I}{2}) \sin(\frac{-h\tilde{\Omega}+hk\omega
I}{2})\Gamma_1/\Gamma_2+(\cdot)
 (\mathrm{i} D)+\cdots, \ \textmd{for}\ \abs{k}>1,
\end{aligned}
\end{equation*}
where
\begin{equation*}
\begin{aligned}\Gamma_1=&\cos(\frac{h\tilde{\Omega}-hk\omega
I}{2})-\cos(\frac{h\tilde{\Omega}+hk\omega
I}{2})+h\tilde{\Omega}\sin(\frac{h\tilde{\Omega}-hk\omega
I}{2}),\\
\Gamma_2=&
\big(-I+\cos(h\tilde{\Omega})+\cos(hk\tilde{\omega})I-\cos(h\tilde{\Omega}+hk\tilde{\omega}I)
+h\tilde{\Omega}\sin(h\tilde{\Omega})-h\tilde{\Omega}\sin(hk\tilde{\omega})I\big)^2\\
&+\big(-h\tilde{\Omega}
\cos(h\tilde{\Omega})+h\tilde{\Omega}\cos(hk\tilde{\omega})I+\sin(h\tilde{\Omega})
+\sin(hk\tilde{\omega})I-\sin(h\tilde{\Omega}+hk\tilde{\omega}I)\big)^2.
\end{aligned}
\end{equation*}
 The operator $(L_3L_4^{-1})(hD)$ has the following Taylor
expansions
\begin{equation*}
\begin{aligned}&(L_3L_4^{-1})(hD)= D-\frac{-2+h\tilde{\Omega}\cot(\frac{h\tilde{\Omega}}{2})}{2h\tilde{\Omega}}(\mathrm{i}hD^2)+\cdots,\\
&(L_3L_4^{-1})(hD-\mathrm{i}h\tilde{\omega})=\mathrm{i}\textmd{diag}\Big(-\tilde{\omega},
-\frac{2\tan(\frac{h\tilde{\omega}}{2})}{h},
\frac{\tilde{\omega}}{-1+h\tilde{\omega}\cot(\frac{h\tilde{\omega}}{2})}
 \Big)+\cdots,\\
&(L_3L_4^{-1})(hD+\mathrm{i}h\tilde{\omega})=\mathrm{i}\textmd{diag}\Big(
 \frac{\tilde{\omega}}{-1+h\tilde{\omega}\cot(\frac{h\tilde{\omega}}{2})},\frac{2\tan(\frac{h\tilde{\omega}}{2})}{h},
\tilde{\omega} \Big)+\cdots,\\
&(L_3L_4^{-1})(hD+\mathrm{i}hk\tilde{\omega})=-
8\tilde{\Omega}\sin(\frac{h \Omega}{2})\sin(\frac{hk\omega
I}{2})\Gamma_1/\Gamma_2\mathrm{i}+(\cdot)
 ( D)+\cdots, \ \textmd{for}\ \abs{k}>1.
\end{aligned}
\end{equation*}
  Moreover, for the operator $L_5(hD,\tau,k)$ and $\abs{k}>0$, it is true that
\begin{equation*}
\begin{aligned}
L_5(hD,
\frac{1}{2},k)=\cos(\frac{kh\tilde{\omega}}{2})+\frac{1}{2}\sin(\frac{kh\tilde{\omega}}{2})(\textmd{i}
hD)+\cdots.
\end{aligned}
\end{equation*}
\end{prop}

\subsection{Modulated Fourier expansion}\label{subs: mfe}

  A modulated Fourier expansion of the EEP method  \eqref{ne exp
integ one-stage} for solving the transformed system
\eqref{necharged-sts-first order} is given as follows.

\begin{theo}\label{energy thm}
Assume that all the conditions of Assumption \ref{ass} hold.  The
numerical solution given by the EEP integrator \eqref{ne exp integ
one-stage} admits a modulated Fourier expansion
\begin{equation}
\begin{aligned} &\tilde{x}_{n}=  \sum\limits_{|k|<N} \mathrm{e}^{\mathrm{i}k\tilde{\omega} t}\tilde{\zeta}^k(t)+\tilde{R}_{h,N}(t),\
\ \tilde{v}_{n}= \sum\limits_{|k|<N}
\mathrm{e}^{\mathrm{i}k\tilde{\omega}
t}\tilde{\eta}^k(t)+\tilde{S}_{h,N}(t),
\end{aligned}
\label{MFE-ERKN}%
\end{equation}
where   $0 \leq t=nh \leq T$, $N$ is a fixed integer such that
\eqref{numerical non-resonance cond} holds and the remainder terms
are bounded by
\begin{equation}
 \tilde{R}_{h,N}(t)=\mathcal{O}(t^2h^{N}),\ \ \ \  \tilde{S}_{h,N}(t)=\mathcal{O}(t^2h^{N-1}).\\
\label{remainder}%
\end{equation}
The bounds of the coefficient functions of $\tilde{\zeta}^k$ as well
as all their derivatives are  given by
\begin{equation}
\begin{aligned} &\dot{\tilde{\zeta}}^0_{\pm1}=\mathcal{O}(\epsilon),\qquad  \quad  \ \ddot{\tilde{\zeta}}^0_{0}=\mathcal{O}(1),\quad \qquad
 \ \ \dot{\tilde{\zeta}}^1_{1}=\mathcal{O}(\epsilon),\quad \ \   \quad \ \ \ \dot{\tilde{\zeta}}^{-1}_{-1}=\mathcal{O}(\epsilon),  \\
&\tilde{\zeta}^{1}_{-1}=\mathcal{O}\big(\frac{\epsilon^3}{\sqrt{h}}\big)\qquad
\   \tilde{\zeta}^{1}_{0}=\mathcal{O}(\epsilon^2),\qquad  \   \ \ \
\
\tilde{\zeta}^{-1}_{1}=\mathcal{O}(\frac{\epsilon^3}{\sqrt{h}}),\quad
\ \ \
\tilde{\zeta}^{-1}_{0}=\mathcal{O}(\epsilon^2),  \\
&\tilde{\zeta}^{k}=\mathcal{O}(\epsilon^{\abs{k} +1}) \quad \ \ \ \
\textmd{for}\ \abs{k}>1,
\end{aligned}
\label{coefficient func1}%
\end{equation}
and further by
\begin{equation}
 \tilde{\zeta}^0_{\pm1}=\mathcal{O}(1),\qquad \quad    \tilde{\zeta}^0_{0}=\mathcal{O}(1),\quad \qquad
\tilde{\zeta}^1_{1}=\mathcal{O}(\epsilon),\quad \ \ \ \
\tilde{\zeta}^{-1}_{-1}=\mathcal{O}(\epsilon).
\label{coefficient func1-1}%
\end{equation}
The bounds of the coefficient functions of $\tilde{\eta}^k$ as well
as all their derivatives are
\begin{equation}
\begin{aligned} &\tilde{\eta}^0_{\pm1}=\mathcal{O}(\epsilon),\qquad \qquad  \quad \ \  \tilde{\eta}^0_{0}=\mathcal{O}(1),\quad
\qquad \ \ \ \tilde{\eta}^1_{1}=\mathrm{i}\tilde{w}
\tilde{\zeta}_{1}^{1}(t)+\mathcal{O}(\epsilon),\quad \ \ \
\tilde{\eta}^{-1}_{-1}=\mathrm{i}\tilde{w} \tilde{\zeta}_{-1}^{-1}(t)+\mathcal{O}(\epsilon),  \\
&\tilde{\eta}^{1}_{-1}=\mathcal{O}\big(\frac{\epsilon^3}{\sqrt{h}\abs{\cos(\frac{h
\tilde{\omega}}{2})}}\big),\ \
\tilde{\eta}^{1}_{0}=\mathcal{O}\big(\frac{\epsilon^2}{\abs{\cos(\frac{h
\tilde{\omega}}{2})}}\big),\ \
\tilde{\eta}^{-1}_{1}=\mathcal{O}\big(\frac{\epsilon^3}{\sqrt{h}\abs{\cos(\frac{h
\tilde{\omega}}{2})}}\big),\ \
\tilde{\eta}^{-1}_{0}=\mathcal{O}\big(\frac{\epsilon^2}{\abs{\cos(\frac{h
\tilde{\omega}}{2})}}\big),  \\
&\tilde{\eta}^{k}=\mathcal{O}(\epsilon^{\abs{k}}) \qquad  \qquad
\qquad \textmd{for}\ \abs{k}>1,
\end{aligned}
\label{coefficient func2}%
\end{equation}
By noticing $P\tilde{\zeta}^{-k}=
\bar{P}\overline{\tilde{\zeta}^{k}},$
$P\tilde{\eta}^{-k}=\bar{P}\overline{\tilde{\eta}^{k}}$ and
  $P^\textup{H}\bar{P}=\left(
                        \begin{array}{ccc}
                          0 & 0 & 1 \\
                          0 & 1 & 0 \\
                          1 & 0 & 0 \\
                        \end{array}
                      \right),
$ one obtains that $\tilde{\zeta}_{-l}^{-k}=
 \overline{\tilde{\zeta}_{l}^{k}}$ and
$\tilde{\eta}_{-l}^{-k}= \overline{\tilde{\eta}_{l}^{k}}$.
 The constants
symbolized by the notation depend on the constants  from Assumption
\ref{ass} and the final time $T$, but are independent of $h$ and
$\tilde{\omega}$.
\end{theo}
\begin{proof}
We will construct the functions
\begin{equation}
\begin{aligned} &\tilde{x}_{h}(t)= \sum\limits_{|k|<N} \mathrm{e}^{\mathrm{i}k\tilde{\omega}
t}\tilde{\zeta}^k(t),\quad
 \ \tilde{v}_{h}(t)= \sum\limits_{|k|<N} \mathrm{e}^{\mathrm{i}k\tilde{\omega}
 t}\tilde{\eta}^k(t)
\end{aligned}
\label{MFE-1}%
\end{equation} with smooth coefficient functions and prove that there is
only a small defect when $\tilde{x}_{h}$ and $\tilde{v}_{h}$ are
inserted into the numerical scheme \eqref{ne exp integ one-stage}.

\vskip1mm  $\bullet$ \textbf{Construction of the coefficients
functions.}
 For the   term $(1-\tau)\tilde{x}_{n}+\tau
\tilde{x}_{n+1}$, we consider the  function
\begin{equation*}
\begin{aligned} &\tilde{q}_{h}(t+\frac{h}{2},\tau)=\sum\limits_{|k|<N}
\mathrm{e}^{\mathrm{i}k\tilde{\omega} (t+\frac{h}{2})}\tilde{\xi}
^k(t+\frac{h}{2},\tau)
\end{aligned}
\end{equation*}
as its modulated Fourier expansion. Then in the light of
\eqref{MFE-1}, we get
\begin{equation*}
\begin{aligned} \tilde{q}_{h}(t+\frac{h}{2},\tau)&=(1-\tau)\sum\limits_{|k|<N} \mathrm{e}^{\mathrm{i}k\tilde{\omega}
t}\tilde{\zeta}^k(t)+\tau \sum\limits_{|k|<N}
\mathrm{e}^{\mathrm{i}k\tilde{\omega} (t+h)}\tilde{\zeta}^k(t+h)\\
&=\sum\limits_{|k|<N} \mathrm{e}^{\mathrm{i}k\tilde{\omega}
(t+\frac{h}{2})}\Big((1-\tau)\mathrm{e}^{-\mathrm{i}k\tilde{\omega}
 \frac{h}{2}}\mathrm{e}^{-
 \frac{h}{2}D} +\tau\mathrm{e}^{\mathrm{i}k\tilde{\omega}
 \frac{h}{2}} \mathrm{e}^{
 \frac{h}{2}D}\Big)\tilde{\zeta}^k(t+\frac{h}{2}),
\end{aligned}
\end{equation*}
which yields
\begin{equation}\label{xip}
\begin{aligned} \tilde{\xi} ^k(t+\frac{h}{2},\tau)=\Big((1-\tau)\mathrm{e}^{-\mathrm{i}k\tilde{\omega}
 \frac{h}{2}}\mathrm{e}^{-
 \frac{h}{2}D} +\tau\mathrm{e}^{\mathrm{i}k\tilde{\omega}
 \frac{h}{2}} \mathrm{e}^{
 \frac{h}{2}D}\Big)\tilde{\zeta}^k(t+\frac{h}{2})=L_5(hD,\tau,k)\tilde{\zeta}^k(t+\frac{h}{2}).
\end{aligned}
\end{equation}

By   eliminating $\mathcal{I}$ in \eqref{ne exp integ one-stage},
one gets
\begin{equation}\label{integ one-stage}
\begin{array}[c]{ll}
\varphi_1(\mathrm{i}h\tilde{\Omega})(\tilde{x}_{n+1}-\tilde{x}_n)=h\varphi_2(\mathrm{i}h\tilde{\Omega})
\tilde{v}_{n+1}+h\big(\varphi^2_1(\mathrm{i}h\tilde{\Omega})-e^{
\mathrm{i}h\tilde{\Omega}}\varphi_2(\mathrm{i}h\tilde{\Omega})\big)\tilde{v}_{n}.
\end{array}
\end{equation}
Inserting  \eqref{MFE-1}   into   \eqref{integ one-stage} and
comparing the coefficients of $\mathrm{e}^{\mathrm{i}k\tilde{\omega}
t}$ yields
\begin{equation}
\begin{aligned} &
\tilde{\eta}^k(t)=(L_3L_4^{-1})(hD+\mathrm{i}kh\tilde{\omega})\tilde{\zeta}^k(t).
\end{aligned}
\label{MFE-33}%
\end{equation}
On the basis of this formula, we can obtain
\begin{equation}
\begin{aligned} &\tilde{\eta}^0(t)=  \dot{\tilde{\zeta}}^0(t)+\mathcal{O}(h),\\
 &\tilde{\eta}_1^1(t)= \mathrm{i}\tilde{w}
\tilde{\zeta}_1^1(t)+\mathcal{O}(h),\ \tilde{\eta}_{-1}^{-1}(t)=
\mathrm{i}\tilde{w} \tilde{\zeta}_{-1}^{-1}(t)+\mathcal{O}(h),
\end{aligned}
\label{MFE-4}%
\end{equation}
which   will be used in the analysis of the next subsection.

 Based on the second formula of  \eqref{ne exp integ one-stage} and \eqref{MFE-33}, one has
\begin{equation*}
\left\{\begin{array}[c]{ll} &L(hD)\tilde{\zeta} ^0(t)=h
\int_{0}^{1}\Big(F(\tilde{\xi} ^0(t,\tau))+
\sum\limits_{s(\alpha)=0}\frac{1}{m!}F^{(m)}(\tilde{\xi}^0(t,\tau))(\tilde{\xi}(t,\tau))^{\alpha}\Big)d\tau,\\
&L(hD+\mathrm{i}kh\omega)\tilde{\zeta} ^k(t)=h \int_{0}^{1}
\sum\limits_{s(\alpha)=k}\frac{1}{m!}F^{(m)}(\tilde{\xi}^0(t,\tau))(\tilde{\xi}(t,\tau))^{\alpha}d\tau,\qquad k\neq0,\\
\end{array}\right.
\end{equation*}
where the sum ranges over $m\geq0$,
$\alpha=(\alpha_1,\ldots,\alpha_m)$ with integer $\alpha_i$
satisfying $0<|\alpha_i|<N$,
$s(\alpha)=\sum\limits_{j=1}^{m}\alpha_j,$ and
$(\tilde{\xi}(t,\tau))^{\alpha}$ is an abbreviation for
$(\tilde{\xi}^{\alpha_1}(t,\tau),\ldots,\tilde{\xi}^{\alpha_m}(t,\tau))$.
This formula as well as \eqref{xip}  presents the modulation system
for the coefficients $\tilde{\zeta} ^k(t)$ of the modulated Fourier
expansion $q_n$. By using the results given in Proposition \ref{lhd
pro} and choosing  the dominate terms in the relations yields the
ansatz of $\tilde{\zeta}^k(t)$:
\begin{equation}\label{ansatz}%
\begin{array}{ll}
\dot{\tilde{\zeta}}^0_{\pm1}(t)=\frac{h}{\mp
 \mathrm{i}h\tilde{\omega}}\big(G_{\pm10}(\cdot)+\cdots\big),\
&\ddot{\tilde{\zeta}}^0_{0}(t)= G_{00}(\cdot)+\cdots,\\
\tilde{\zeta}_{-1}^1(t)=\frac{h\tilde{\omega}\cos(\frac{h\tilde{\omega}}{2})-\sin(\frac{h\tilde{\omega}}{2})}{-\tilde{\omega}^2\sin(h\tilde{\omega})}
\big(F^1_{-10}(\cdot)+\cdots\big),\
&\tilde{\zeta}_{0}^1(t)=\frac{\textmd{sinc}(\frac{h\tilde{\omega}}{2})}{\tilde{\omega}}
\big(F^1_{00}(\cdot)+\cdots\big),
\\
\dot{\tilde{\zeta}}_{1}^1(t)=\frac{\frac{1}{2}h}{\mathrm{i}\tan(\frac{1}{2}h\tilde{\omega})}
\big(F^1_{10}(\cdot)+\cdots\big),\
&\dot{\tilde{\zeta}}_{-1}^{-1}(t)=\frac{\frac{1}{2}h}{\mathrm{i}\tan(\frac{1}{2}h\tilde{\omega})}
\big(F^{-1}_{-10}(\cdot)+\cdots\big),\\
\tilde{\zeta}_{0}^{-1}(t)=\frac{\textmd{sinc}(\frac{h\tilde{\omega}}{2})}{\tilde{\omega}}
\big(F^{-1}_{00}(\cdot)+\cdots\big), \
&\tilde{\zeta}_{1}^{-1}(t)=\frac{h\tilde{\omega}\cos(\frac{h\tilde{\omega}}{2})-\sin(\frac{h\tilde{\omega}}{2})}{-\tilde{\omega}^2\sin(h\tilde{\omega})}
 \big(F^{-1}_{10}(\cdot)+\cdots\big),\\
\tilde{\zeta}^{k}(t)= \frac{h \Gamma_2}{-  8h\tilde{\Omega}^2
\sin(\frac{hk\omega I}{2}) \sin(\frac{-h\tilde{\Omega}+hk\omega
I}{2})\Gamma_1} \big(F_0^{k}(\cdot)+\cdots\big) &\textmd{for}\
\abs{k}>1.
\end{array} %
\end{equation}
 Similarly, the ansatz
of the modulated Fourier functions $\tilde{\eta}^k(t)$ can be
obtained by considering  the dominating terms of
  \eqref{MFE-33} and \eqref{ansatz}.


  \vskip1mm  $\bullet$
\textbf{Initial values.}

According to the conditions $x_{h}(0)=\tilde{x}_0$ and
 $v_{h}(0)=\tilde{v}_0$,
we obtain
\begin{equation}\label{Initial values-1}%
\begin{aligned}
&\tilde{x}_0^0=\tilde{\zeta}_0^0(0)+\mathcal{O}(\epsilon),\\
&\tilde{x}_{\pm1}^0=\tilde{\zeta}_{\pm1}^0(0)+\tilde{\zeta}_{\pm1}^1(0)+\mathcal{O}(\epsilon^2),\\
&\tilde{v}_0^0=\tilde{\eta}_0^0(0)+\mathcal{O}(h)=\dot{\tilde{\zeta}}_0^0(0)+\mathcal{O}(\epsilon),\\
&\tilde{v}_{1}^0=\tilde{\eta}_{1}^0(0)+\tilde{\eta}_{1}^{1}(0)+\mathcal{O}(\frac{\epsilon^2}{\sqrt{h}})=
\dot{\tilde{\zeta}}_{1}^0(0)+\mathrm{i}\tilde{w}\tilde{\zeta}_{1}^{1}(0)+\mathcal{O}(\epsilon),\\
&\tilde{v}_{-1}^0=\tilde{\eta}_{-1}^0(0)+\tilde{\eta}_{-1}^{-1}(0)+\mathcal{O}(\frac{\epsilon^2}{\sqrt{h}})=
\dot{\tilde{\zeta}}_{-1}^0(0)+\mathrm{i}\tilde{w}\tilde{\zeta}_{-1}^{-1}(0)+\mathcal{O}(\epsilon). \end{aligned} %
\end{equation}
Thus  the initial values $\tilde{\zeta}_0^0(0)=\mathcal{O}(1)$ and
$\dot{\tilde{\zeta}}_0^0(0)=\mathcal{O}(1)$ can be arrived by
considering the first and third formulae. It follows from  the
fourth formula that $\tilde{\zeta}_{1}^{1}(0)=\frac{1}{\mathrm{i}
\tilde{\omega}}\big(\tilde{v}_{1}^0-\dot{\tilde{\zeta}}_{1}^0(0)+\mathcal{O}(\epsilon)\big)
=\mathcal{O}(\epsilon)$ and similarly one has
$\tilde{\zeta}_{-1}^{-1}(0)=\mathcal{O}(\epsilon)$. Then  one gets
the initial value $\tilde{\zeta}_{\pm1}^0(0)=\mathcal{O}(1)$  in the
light of
  \eqref{Initial values-1}.

\vskip1mm  $\bullet$ \textbf{Bounds of the coefficients functions.}

The bounds \eqref{coefficient func1} can be obtained by considering
the initial values obtained above, the ansatz
  and Assumption \ref{ass}. On the basis of the initial values obtained above, we  get the bounds \eqref{coefficient func1-1}. The bounds given
in \eqref{coefficient func2} are true in the light of \eqref{MFE-4}.

\vskip1mm  $\bullet$ \textbf{Defect.}

The defect \eqref{remainder} can be obtained   by using the
Lipschitz continuous of the nonlinearity  and  the standard
convergence estimates.

\end{proof}

\begin{theo}\label{energy thm2}
The numerical solution given by the  EEP integrator \eqref{EAVF} for
solving  \eqref{charged-sts-first order} admits the following
modulated Fourier expansion
\begin{equation}
\begin{aligned} &x_{n}=  \sum\limits_{|k|<N} \mathrm{e}^{\mathrm{i}k\tilde{\omega} t}\zeta^k(t)+R_{h,N}(t),\
\ v_{n}= \sum\limits_{|k|<N} \mathrm{e}^{\mathrm{i}k\tilde{\omega}
t}\eta^k(t)+S_{h,N}(t),
\end{aligned}
\label{MFE-ERKN}%
\end{equation}
where  $\zeta^k(t)=P\tilde{\zeta}^k(t)$ and
$\eta^k(t)=P\tilde{\eta}^k(t).$ This relation yields
$\zeta^{-k}=\overline{\zeta^{k}}$ and
 $\eta^{-k}=\overline{\eta^{k}}$.  The bounds of the remainders $R_{h,N}$ and $S_{h,N}$  are the same as
 those given in Theorem \ref{energy thm}.
\end{theo}

\subsection{An almost-invariant} \label{subs: ai}
With the coefficient functions constructed in last subsection, we
let
$$\vec{\tilde{\zeta}}=\big(\tilde{\zeta}^{-N+1}(t),\cdots,\tilde{\zeta}^{-1}(t),
\tilde{\zeta}^{0}(t),\tilde{\zeta}^{1}(t),\cdots,\tilde{\zeta}^{N-1}(t)\big)$$
and
$$\vec{\tilde{\eta}}=\big(\tilde{\eta}^{-N+1}(t),\cdots,\tilde{\eta}^{-1}(t),
\tilde{\eta}^{0}(t),\tilde{\eta}^{1}(t),\cdots,\tilde{\eta}^{N-1}(t)\big).
$$
An almost-invariant of the modulated Fourier expansion
\eqref{MFE-ERKN} is given as follows.

\begin{theo}\label{2 invariant thm}
Under the conditions of Theorem \ref{energy thm}, there exists a
function
$\widehat{\mathcal{M}}[\vec{\tilde{\zeta}},\vec{\tilde{\eta}}]$ such
that
\begin{equation*}
\widehat{\mathcal{M}}[\vec{\tilde{\zeta}},\vec{\tilde{\eta}}](t)=\widehat{\mathcal{M}}[\vec{\tilde{\zeta}},\vec{\tilde{\eta}}](0)+\mathcal{O}(th^{N})
\end{equation*}
for $0\leq t\leq T,$ where
\begin{equation}\begin{aligned}
\widehat{\mathcal{M}}[\vec{\tilde{\zeta}},\vec{\tilde{\eta}}]&=\frac{1}{2}\frac{\frac{1}{2}
h\tilde{\omega}^3\cos(\frac{1}{2} h\tilde{\omega})}{\sin(\frac{1}{2}
h\tilde{\omega})} \big(\abs{ \tilde{\zeta}_1^1}^2+\abs{
\tilde{\zeta}_{-1}^{-1}}^2\big)+\mathcal{O}(h)\\
&=\frac{1}{2}\frac{\frac{1}{2} h\tilde{\omega}\cos(\frac{1}{2}
h\tilde{\omega})}{\sin(\frac{1}{2} h\tilde{\omega})} \big(\abs{
\tilde{\eta}_1^1}^2+\abs{
\tilde{\eta}_{-1}^{-1}}^2\big)+\mathcal{O}(h).
\label{mm invariant}%
\end{aligned}
\end{equation}
\end{theo}

\begin{proof}
With  Theorem \ref{energy thm}, we obtain that
\begin{equation*}
\begin{aligned}
& L(hD) \tilde{x}_{h}(t)=h \int_{0}^{1}
\tilde{F}(\tilde{q}_{h}(t,\tau))d\tau+\mathcal{O}(h^{N+2}),
\end{aligned}
\end{equation*}
where we have used  the   denotations
$\tilde{x}_{h}(t)=\sum\limits_{ |k|<N}\tilde{x}^k_{h}(t),\
\tilde{q}_{h}(t,\tau)=\sum\limits_{ |k|<N}\tilde{q}^k_{h}(t,\tau)$
with  $ \tilde{x}^k_{h}(t)=\mathrm{e}^{\mathrm{i}k\tilde{\omega}
t}\tilde{\zeta}^k(t)$ and $
\tilde{q}^k_{h}(t,\tau)=\mathrm{e}^{\mathrm{i}k\tilde{\omega}
t}\tilde{\xi}^k(t,\tau).$ Multiplication of this result with $P$
yields
\begin{equation*}
\begin{aligned}
& PL(hD)P^\textup{H} P\tilde{x}_{h}(t)=PL(hD)P^\textup{H}
x_{h}(t)\\
&=h\int_{0}^{1}P \tilde{F}(\tilde{q}_{h}(t,\tau))d\tau
 +\mathcal{O}(h^{N+2})=h\int_{0}^{1}F(q_{h}(t,\tau))d\tau+\mathcal{O}(h^{N+2}),
\end{aligned}
\end{equation*}
where   $x_{h}(t)=\sum\limits_{ |k|<N}x^k_{h}(t)$  with
$x^k_{h}(t)=\mathrm{e}^{\mathrm{i}k\tilde{\omega} t}\zeta^k(t) $ and
$q_{h}(t,\tau)=\sum\limits_{ |k|<N}q_{h}^k(t,\tau)$  with
$q_{h}^k(t,\tau)=\mathrm{e}^{\mathrm{i}k\tilde{\omega}
t}\xi^k(t,\tau).$   Rewrite the equation in terms of $x^k_{h}$ and
then we get
\begin{equation*}
\begin{aligned} &PL(hD)P^\textup{H} x^k_{h}(t)=-h\nabla_{x^{-k}}\mathcal{U}(\vec{q}(t,\tau))+\mathcal{O}(h^{N+2}),
\end{aligned}
\end{equation*}
where $\mathcal{U}(\vec{q}(t,\tau))$ is defined as
\begin{equation}
\begin{aligned}
&\mathcal{U}(\vec{q}(t,\tau))=\int_{0}^{1}U(q_h^0(t,\tau))d\tau+
\sum\limits_{s(\alpha)=0}\int_{0}^{1}\frac{1}{m!}U^{(m)}(q_h^0(t,\tau))
(q_h(t,\tau))^{\alpha}d\tau
\end{aligned}
\label{newuu}%
\end{equation}
with
$$\vec{q}(t,\tau)=\big(q_h^{-N+1}(t,\tau),\cdots,q_h^{-1}(t,\tau),
q_h^{0}(t,\tau),q_h^{1}(t,\tau),\cdots,q_h^{N-1}(t,\tau)\big).$$
Define the  vector function
$$\vec{q}(\lambda,t,\tau)=\big( \mathrm{e}^{\mathrm{i}(-N+1)\lambda
\tilde{\omega}}q^{-N+1}_h(t,\tau),\cdots,
q^{0}_h(t,\tau),\cdots,\mathrm{e}^{\mathrm{i}(N-1)\lambda
\tilde{\omega}}q^{N-1}_h(t,\tau)\big).$$   We have the invariance
property
 that $\mathcal{U}( \vec{q}(\lambda,t,\tau))$ is  independent of
$\lambda$ and $\tau$. Thus, differentiation with respect $\lambda$
implies
\begin{equation*}
\begin{aligned}0=& \frac{\partial}{\partial\lambda}  \mathcal{U}(
\vec{q}(\lambda,t,\tau))  =\Big(\frac{\partial}{\partial q}
 \mathcal{U}( \vec{q}(\lambda,t,\tau)) \Big)^\intercal
\frac{\partial}{\partial\lambda }\vec{q}(\lambda,t,\tau)\\ =&
\sum\limits_{|k|<N}\mathrm{i}k\tilde{\omega}\mathrm{e}^{\mathrm{i}k\lambda
\tilde{\omega}} (q^{k}_h(\lambda,t,\tau))^\intercal\nabla_{k}
  \mathcal{U}( \vec{q}(\lambda,t,\tau)).\end{aligned}
\end{equation*}
By letting $\lambda=0$ and $\tau=\frac{1}{2}$, we obtain $
\sum\limits_{|k|<N}\mathrm{i}k\tilde{\omega}
(q^{k}_h(t,\frac{1}{2}))^\intercal \nabla_{ k} \mathcal{U}(
\vec{q}(t,\frac{1}{2}))=0. $ Consequently, we have
\begin{equation}
\begin{aligned}
0=&\sum\limits_{|k|<N}\mathrm{i}k\tilde{\omega}
(q^{-k}_h(t,\frac{1}{2}))^\intercal \nabla_{ -k}
\mathcal{U}( \vec{q}(t,\frac{1}{2}))\\
=& \sum\limits_{|k|<N}\mathrm{i}k\tilde{\omega}
(q^{-k}_h(t,\frac{1}{2}))^\intercal\frac{1}{-h }
PL(hD)P^\textup{H} x^k_{h}(t)+\mathcal{O}(h^{N})\\
=& \sum\limits_{|k|<N}\mathrm{i}k\tilde{\omega}
(\bar{q^{k}_h}(t,\frac{1}{2}))^\intercal\frac{1}{-h }
PL(hD)P^\textup{H} x^k_{h}(t)+\mathcal{O}(h^{N})\\
= & \sum\limits_{|k|<N}
 \frac{\mathrm{i}k\tilde{\omega}}{-h}\big(\bar{\xi}^k(t,\frac{1}{2})\big)^\intercal PL(hD+\mathrm{i}hk\tilde{\omega})P^\textup{H}
 \zeta^k(t)+\mathcal{O}(h^{N}).
\end{aligned}
\end{equation}
According to the relation of the coefficient functions  given in
Theorem \ref{energy thm2}, we obtain
\begin{equation}
\begin{aligned}
\mathcal{O}(h^{N})
 = & \sum\limits_{|k|<N}
 \frac{\mathrm{i}k\tilde{\omega}}{-h}\big(\overline{\tilde{\xi}^{k}}(t,\frac{1}{2})\big)^\intercal P^\textup{H} P
 L(hD+\mathrm{i}hk\tilde{\omega})P^\textup{H} P
\tilde{\zeta}^k(t)\\
 = & \sum\limits_{|k|<N}
 \frac{\mathrm{i}k\tilde{\omega}}{-h}\big(\overline{\tilde{\xi}^{k}}(t,\frac{1}{2})\big)^\intercal
 L(hD+\mathrm{i}hk\tilde{\omega})
\tilde{\zeta}^k(t)\\
 = & \sum\limits_{|k|<N}
 \frac{\mathrm{i}k\tilde{\omega}}{-h}\big(L_5(hD,\frac{1}{2},-k)\overline{\tilde{\zeta}^{k}}(t,\frac{1}{2})\big)^\intercal
 L(hD+\mathrm{i}hk\tilde{\omega})
\tilde{\zeta}^k(t).
\end{aligned}\label{duu-I}
\end{equation}
By Proposition \ref{lhd pro}    and   the following  the ``magic
formulas" on p. 508 of \cite{hairer2006},
 it can be verified that the right-hand side of  \eqref{duu-I}  is a total
derivative. This proves that there exists a function
$\widehat{\mathcal{M}}$ such
 that
$\frac{d}{dt}\widehat{\mathcal{M}}[\vec{\tilde{\zeta}},\vec{\tilde{\eta}}]=\mathcal{O}(h^{N})$.
  The first statement follows by integration this result.

Based on the bounds of the coefficients functions given in Theorem
\ref{energy thm}, the   $\widehat{\mathcal{H}}$ can be expressed as
\begin{equation*}
\begin{aligned}\widehat{\mathcal{M}}[\vec{\tilde{\zeta}},\vec{\tilde{\eta}}]=&\frac{1}{2}\frac{\frac{1}{2} h\tilde{\omega}^3\cos(\frac{1}{2}
h\tilde{\omega})}{\sin(\frac{1}{2} h\tilde{\omega})} \big(\abs{
\tilde{\zeta}_1^1}^2+\abs{
\tilde{\zeta}_{-1}^{-1}}^2\big)+\mathcal{O}(h^2)\\
=&\frac{1}{2}\frac{\frac{1}{2} h\tilde{\omega}\cos(\frac{1}{2}
h\tilde{\omega})}{\sin(\frac{1}{2} h\tilde{\omega})} \big(\abs{
\tilde{\eta}_1^1}^2+\abs{
\tilde{\eta}_{-1}^{-1}}^2\big)+\mathcal{O}(h),
\end{aligned}
\end{equation*}
where the second formula of \eqref{MFE-4} was used.   This implies
the second statement of the theorem.

\end{proof}

\subsection{Long time  magnetic moment conservation}\label{sec:Long-time near-conservation}
From  the bounds presented in Theorem \ref{energy thm} and from the
expression \eqref{newmomentum for B} it follows that
 \begin{equation*}
 \begin{aligned} I(x^n,v^n)= \tilde{I}(\tilde{x}^n,\tilde{v}^n)&= \frac{1}{2 \abs{\tilde{B}}}  \big(\abs{
\tilde{\eta}_1^1}^2+\abs{ \tilde{\eta}_{-1}^{-1}}^2\big)
+\mathcal{O}(h).
\end{aligned}\end{equation*}
Looking closing at this formula and \eqref{mm invariant}, we get the
following   relationship between  the magnetic moment and  the
  almost-invariant $\widehat{\mathcal{M}}$:
\begin{equation*}  \tilde{I}(\tilde{x}^n,\tilde{v}^n)=\frac{1}{ \abs{ \tilde{B}}}\frac{\tan(\frac{1}{2}
h\tilde{\omega})}{\frac{1}{2}
h\tilde{\omega}}\widehat{\mathcal{M}}[\vec{\tilde{\zeta}},\vec{\tilde{\eta}}]
+\mathcal{O}\Big(\frac{1}{\abs{\cos(\frac{1}{2}
h\tilde{\omega})}}h\Big)+\mathcal{O}(h).
 \end{equation*}

  Based on the above analysis and following the
way used in Chapter XIII of \cite{hairer2006}, Theorem \ref{2 sym
Long-time thm} is easily proved by   patching together the local
near-conservation result.

\section{Conclusions} \label{sec:conclusions}
In this paper, we have studied exponential  energy-preserving
methods   for charged-particle dynamics in a  strong and constant
magnetic field. It was shown that this method can exactly preserve
the energy of the charged-particle dynamics. It is worth mentioning
that the long-time magnetic moment conservation was also been
researched by deriving  a modulated Fourier expansion of the method
and showing  an almost-invariant of the modulation system.  In this
paper, we have also studied other properties of the method. The
effectiveness of the method  is emphasized by carrying out a
numerical experiment.

\section*{Acknowledgements}
The author is grateful to Professor Christian Lubich for drawing my
attention to the charged-particle dynamics and the long-term
analysis of energy-preserving methods.

\end{document}